\documentclass[a4paper]{article}
\usepackage{geometry}
\usepackage{amssymb}
\usepackage{latexsym}
\usepackage{amsmath}
\usepackage{indentfirst}
\usepackage{graphicx}
 \usepackage[colorlinks=true]{hyperref}
\usepackage{mathrsfs} 

\newtheorem{Thm}{Theorem}[section]
\newtheorem{Cor}{Corollary}[section]
\newtheorem{Lem}{Lemma}[section]

\newtheorem{Rek}{Remark}[section]

\newcommand{\R}{\mathbb{R}}

\numberwithin{equation}{section} \numberwithin{figure}{section}

\newenvironment{proof}{\medskip\par\noindent{\bf Proof\/}:}{\qquad
\raisebox{-0.5mm}{\rule{1.5mm}{1mm}}\vspace{6pt}}

\begin{document}
\title{\Large\bf Existence and concentration of solution for Schr\"odinger-Poisson system with local potential}
\author{
{Zhipeng Yang$^{1}$}\thanks{zhipeng.yang@mathematik.uni-goettingen.de},~~
{Yuanyang Yu$^{2,3}$}\thanks{Corresponding author: yuyuanyang18@mails.ucas.ac.cn}\\
\small Mathematical Institute, Georg-August-University of G\"ottingen, G\"ottingen 37073, Germany.$^{1}$\\
\small Institute of Mathematics, Academy of Mathematics and Systems Science\\
\small Chinese Academy of Sciences, Beijing 100190,  P.R.China$^{2}$\\
\small University of Chinese Academy of Sciences, Beijing 100049,  P.R.China$^{3}$\\
}

\date{}

\date{} \maketitle

\noindent{\bf Abstract:}
In this paper, we study the following nonlinear Schr\"odinger-Poisson type equation
\begin{equation*}
\begin{cases}
-\varepsilon^2\Delta u+V(x)u+K(x)\phi u=f(u)&\text{in}\ \mathbb{R}^3,\\
-\varepsilon^2\Delta \phi=K(x)u^2&\text{in}\ \mathbb{R}^3,
\end{cases}
\end{equation*}
where $\varepsilon>0$ is a small parameter, $V: \mathbb{R}^3\rightarrow \mathbb{R}$ is a continuous potential and  $K: \mathbb{R}^3\rightarrow \mathbb{R}$ is used to describe the electron charge. Under suitable assumptions on $V(x), K(x)$ and $f$, we prove existence and concentration properties of ground state solutions for $\varepsilon>0$ small. Moreover, we summarize some open problems for the Schr\"odinger-Poisson system.
\par
\vspace{6mm} \noindent{\bf Keywords:} Ground state solution, concentration, Schr\"odinger-Poisson system.

\vspace{1mm} \noindent{\bf AMS subject classification:} 35A15, 35B40, 35J20.

\section{Introduction and main results}
In this paper, we consider the following Schr\"odinger-Poisson system
\begin{equation}\label{eq1.1}
\begin{cases}
-\varepsilon^2\Delta u+V(x)u+K(x)\phi u=f(u)&\text{in}\ \mathbb{R}^3,\\
-\varepsilon^2\Delta \phi=K(x)u^2&\text{in}\ \mathbb{R}^3,
\end{cases}
\end{equation}
where $\varepsilon>0$ is a small parameter, $V(x)$ is a potential function and $K(x)$ is used to describe the electron charge.
We are interested in the existence of positive ground state solutions for competitive relationship between $V(x)$ and $K(x)$, and their asymptotical behavior as $\varepsilon\to0.$ This problem arises when one
studies the interaction between an unknown electromagnetic field and the nonlinear Schr\"odinger field
\begin{equation}\label{eq1.2}
\begin{cases}
i\hbar\frac{\partial \Psi(x,t)}{\partial t}=-\frac{\hbar^2}{2m}\Delta\Psi+\widetilde{V}(x)\Psi+\phi\Psi-f(|\Psi|)& \text{in}\  \mathbb{R}^3\times\mathbb{R},\\
-\Delta\phi=|\Psi|^2& \text{in}\ \mathbb{R}^3,
\end{cases}
\end{equation}
where $i$ is the imaginary unit, $\hbar$ is the Planck constant, $m$ is the mass of the field $\Psi$, $\widetilde{V}(x)$ is a given external potential and $f(exp(i\theta)\xi)=exp(i\theta)f(\xi)$ for $\theta,\xi\in\mathbb{R}$ is a nonlinear function which describes the interaction among many particles.
\par
In recent years, a great deal of work has been devoted to the study of standing waves, i.e.
\[\Psi(x,t)=e^{-iEt}u(x),\]
which leads to the stationary Schr\"odinger-Poisson system \eqref{eq1.1} with $\varepsilon^2=\frac{\hbar^2}{2m}$, $\widetilde{V}(x)=V(x)-E$. When $\varepsilon=V(x)=K(x)=1$, \eqref{eq1.1} was first proposed by Benci and Fortunato \cite{Benci-Fortunato1998TMNA} in 1998 on a bounded domain, and is related to the Hartree equation (\cite{Lions1987}). Then, they \cite{BenciNA01,BenciRMY02} continued to study the model
\begin{equation*}
\begin{cases}
-\Delta u+u+\phi u=|u|^{p-2}u& \text{in}\ \mathbb{R}^3,\\
-\Delta \phi=u^2& \text{in}\ \mathbb{R}^3,
\end{cases}
\end{equation*}
for $4<p<6$.
D'Aprile and Mugnai \cite{DAprile-Mugnai2004PRSESA} studied the existence of radially symmetric solitary waves with $4 \leq p<6$ and obtained solutions as Mountain-Pass critical points for the associated energy functional. Azzollini et al. \cite{Azzollini-d'Avenia-Pomponio2010AIHPAN} proved the existence of nontrivial solutions by supposing that the nonlinearity satisfies Berestycki-Lions type assumptions. Ruiz \cite{Ruiz2006}, Ambrosetti and Ruiz \cite{Ambrosetti-Ruiz2008CCM} considered the system with parameter, that is
\begin{equation*}
\begin{cases}
-\Delta u+u+\mu \phi u=|u|^{p-2} u& \text{in}\ \mathbb{R}^3,\\
-\Delta \phi=4 \pi u^{2}& \text{in}\ \mathbb{R}^3,
\end{cases}
\end{equation*}
By working in the radial functions subspace of $H^{1}\left(\mathbb{R}^{3}\right),$ they obtained the existence and multiplicity results depending on the parameter $p$ and $\mu>0$.
\par

If the potential $V(x)$ is not a constant, there are also some works about Schr\"odinger-Poisson system
\begin{equation}\label{eq1.3}
\begin{cases}
-\Delta u+V(x)u+\phi u=f(u)&\text{in}\ \mathbb{R}^3,\\
-\Delta \phi=u^2&\text{in}\ \mathbb{R}^3.
\end{cases}
\end{equation}
Since $V(x)$ might not be radial, one cannot work in the radial functions space directly, then we must look for some conditions on $V(x)$ to overcome the lack of compactness. Wang and Zhou \cite{WangZhou2007} assumed that
the potential $V$ satisfies the global condition
\begin{itemize}
\item[$(V)$] $0<\inf\limits_{x\in\R^N}V(x)<\liminf\limits_{|x|\rightarrow\infty}V(x)=V_{\infty}$,
\end{itemize}
which was introduced by Rabinowitz \cite{Rabinowitz1992},
they got the existence and non-existence of positive solutions for the problem \eqref{eq1.3} with asymptotically nonlinearities.
Jiang and Zhou \cite{Jiang-Zhou11JDE} first applied the steep potential well conditions to Schr\"{o}dinger-Poisson system, and proved the existence of solutions. Moreover, they also studied the asymptotic behavior of solutions by combining domains approximation with priori estimates. Later, Zhao et al. \cite{Zhao-Liu-Zhao2013JDE} considered a case allowing the potential $V$ change sign. Using variational setting of \cite{Ding-Wei07JFA}, they obtained the existence and asymptotic behavior of nontrivial solutions for $p\in(3,6)$.
If the electronic potential $K(x)$ is not a constant, Ambrosetti \cite{Ambrosetti2008MJM} studied the existence with a parameter in front of the nonlocal term. Cerami and Varia \cite{Cerami-Vaira2010JDE} studied
\begin{equation*}
\begin{cases}
-\Delta u+u+K(x) \phi u=a(x)|u|^{p-2} u& \text{in}\ \mathbb{R}^3,\\
-\Delta \phi=K(x) u^{2}& \text{in}\ \mathbb{R}^3.
\end{cases}
\end{equation*}
Under suitable assumptions on $K(x)$ and $a(x)$, the authors proved the existence of positive solutions. If the
potential is radial symmetry and has asymptotical behavior at infinity, Li et al. \cite{LiPengYan10CCM} studied the existence of infinitely many sign-changing solutions. For other related nonlocal variational problems, we refer the interested readers to see \cite{Azzollini2010JDE,Chen20JDE,jinyang,LiLiShi17CV,LiuMosconiJDE20,Schaftingen16CV,SunWu20JDE,yang19,yang21} and the references therein.

\par
If we consider the case that the parameter $\varepsilon$ goes to zero, the problem is used to describe the transition between the Quantum Mechanics and the Classical Mechanics. The study of a single Schr\"odinger equation
\[
-\varepsilon^{2} \Delta u+V(x) u=g(u), \text { in } \mathbb{R}^{N}
\]
goes back to the pioneer work \cite{Floer86JFA} by Floer and Weinstein. In \cite{D'Aprile-Wei2005SIAMJMA}, D'Aprile and Wei considered
\begin{equation}\label{eq1.4}
\begin{cases}
-\varepsilon^{2} \Delta u+V(x) u+\phi u=|u|^{p-2} u& \text{in}\ \mathbb{R}^{3},\\
-\Delta \phi=u^2& \text{in}\ \mathbb{R}^{3},
\end{cases}
\end{equation}
for $p \in\left(1, \frac{11}{7}\right)$, they constructed a family of positive radially symmetric solutions concentrating around a sphere when the potential is constant.
When the potential $V(x)$ satisfies the global condition $(V)$, He \cite{He2011zamp} studied the multiplicity of positive solutions and proved that these positive solutions concentrate around the global minimum of the potential $V$.  Wang et al. \cite{Wang13CV} studied the existence  and the concentration behavior of ground state solutions for a subcritical problem with competing potentials. The critical case was considered in \cite{He-Zou2012JMP}, He and Zou proved that system \eqref{eq1.4} possesses a positive ground state solution which concentrate around the global minimum of $V$.
For other related results, we may refer the readers to \cite{Ruiz05MMM,YangNA20,YangDing09NA,Zhang14JMP,ZhangXia18JDE} for recent progress.
\par
Motivated by the results mentioned above, the aim of this paper is to continue to study the existence and
concentration of solutions for the Schr\"odinger-Poisson system. In fact, we are interested in the following
Schr\"odinger-Poisson system
\begin{equation*}
\begin{cases}
-\varepsilon^2\Delta u+V(x)u+K(x)\phi u=f(u)&\text{in}\ \mathbb{R}^3,\\
-\varepsilon^2\Delta \phi=K(x)u^2&\text{in}\ \mathbb{R}^3,
\end{cases}
\end{equation*}
where $\varepsilon>0$ is a small parameter, $V: \mathbb{R}^3\rightarrow \mathbb{R}$ is a continuous potential and  $K: \mathbb{R}^3\rightarrow \mathbb{R}$ is used to describe the electron charge which satisfy:
\begin{itemize}
\item[$(V_1)$] $V\in C(\R^3,\R)$ and $\inf\limits_{x\in\R^3}V(x)=V_1>0$.
\item[$(V_2)$] There is an open and bounded domain $\Lambda$ such that
\begin{equation*}
0<V_0:=\inf_{\Lambda}V(x)<\min_{\partial\Lambda}V(x).
\end{equation*}
Without loss of generality we assume $0\in \Lambda$. Under the assumption
\item[$(K)$] $K$ is continuous and bounded with $K(x)\geq 0, K\not\equiv 0,~~K(x)=0~\text{if}~x\in \mathcal{M}$,
where we assume that $\mathcal{M}=\{x\in\Lambda: V(x)=V_0\}\neq\emptyset$ and $V(0)=\min\limits_{x\in \Lambda}V(x)=V_0$.
\end{itemize}

For the nonlinear term, we assume:
\begin{itemize}
\item[$(f_1)$]  $f\in C(\mathbb{R},\mathbb{R}),f(t)=o(t^3)$ as $t\rightarrow 0$ and $f(t)=0$ for all $t\leq 0$;
\item[$(f_2)$]  there is $4<p<6$ such that
\begin{equation*}
|f(t)|\leq c_1(1+|t|^{p-1})
\end{equation*}
for all $t\in \mathbb{R}$ and some $c_1>0$;
\item[$(f_3)$]  there is a constant $\mu>4$ such that
\begin{equation*}
0<\mu F(t)\leq f(t)t,~\forall>0,
\end{equation*}
where $F(t)=\int_0^tf(\tau)d\tau$;
\item[$(f_4)$]  The function $\frac{f(t)}{t^3}$ is strictly increasing for $t>0$.
\end{itemize}
\begin{Rek}\label{Rek1.1}
Note that, $(f_1)$ and $(f_2)$ imply that for each $\epsilon>0$, there is $C_\epsilon>0$ such that
\begin{equation*}
f(t)\leq \epsilon t+C_\epsilon t^{p-1}~~~\text{and}~~~F(t)\leq\epsilon t^2+C_\epsilon t^p,~~\forall t\geq 0.
\end{equation*}
By $(f_3)$, we deduce that
\begin{equation*}
F(t)>0~~~\text{and}~~~\frac{1}{4}f(t)t^2-F(t)>0,~~\forall t>0.
\end{equation*}
Moreover, it follows from $(f_4)$ that there exist $C_1,C_2>0$ such that
\begin{equation*}
F(t)\geq C_1t^\mu-C_2,~\forall t\geq 0.
\end{equation*}
\end{Rek}
\par
In this case, we want to study how the behavior of the potential function and electric field will affect
the existence and concentration of the ground state solutions. Now we state our main results as follows.

\begin{Thm}\label{Thm1.1}
 Assume that $(V_1)$, $(V_2)$, $(K)$ and $(f_1)$-$(f_4)$ hold. Then for any $\varepsilon>0$ small:
\begin{itemize}
\item[$(i)$] The system \eqref{eq1.1} has a ground state solution $\omega_\varepsilon$;
\item[$(ii)$]   $\omega_\varepsilon$ possesses a global maximum point $x_\varepsilon$ such that
\begin{equation*}
\lim_{\varepsilon\rightarrow 0}V(x_\varepsilon)=V_0;
\end{equation*}
\item[$(iii)$] Setting $v_\varepsilon(x):=\omega_\varepsilon(\varepsilon x+x_\varepsilon)$, then $v_\varepsilon$ converges to a ground state solution of
\begin{equation*}
-\Delta u+V_0u=f(u),\quad \text{in}~\mathbb{R}^3;
\end{equation*}
\item[$(iv)$] There are positive constants $C_1,C_2$ independent of $\varepsilon$ such that
\begin{equation*}
\omega_\varepsilon(x)\leq Ce^{-\frac{c}{\varepsilon}|x-x_\varepsilon|},~\forall x\in \mathbb{R}^3.
\end{equation*}
\end{itemize}
\end{Thm}

\begin{Cor}\label{Cor1.1}
 Assume that $(V_1)$ holds. If there exist multiple disjoint bounded domains $\Lambda_j,j = 1,\cdots,k$ and constants $c_1 < c_2 <\cdots< c_k$ such that
\begin{equation*}
c_j:=\inf_{\Lambda_j}V(x)<\min_{\partial \Lambda_j}V(x).
\end{equation*}
Then for any $\varepsilon>0$ small:
\begin{itemize}
\item[$(i)$] The system \eqref{eq1.1} has a ground state solution $\omega_\varepsilon^j$;
\item[$(ii)$]   $\omega_\varepsilon^j$ possesses a global maximum point $x_\varepsilon^j$ such that
\begin{equation*}
\lim_{\varepsilon\rightarrow 0}V(x_\varepsilon^j)=c_j;
\end{equation*}
\item[$(iii)$] Setting $v_\varepsilon^j(x):=\omega_\varepsilon^j(\varepsilon x+x_\varepsilon)$, then $v_\varepsilon^j$ converges in to a ground state solution of
\begin{equation*}
-\Delta u+c_ju=f(u),\quad \text{in}~\mathbb{R}^3;
\end{equation*}
\item[$(iv)$] There are positive constants $C_1,C_2$ independent of $\varepsilon$ such that
\begin{equation*}
\omega_\varepsilon^j(x)\leq Ce^{-\frac{c}{\varepsilon}|x-x_\varepsilon^j|},~\forall x\in \mathbb{R}^3.
\end{equation*}
\end{itemize}
\end{Cor}
\begin{Rek}\label{Rek1.2}
We remark here that in Corollary \ref{Cor1.1}, the solutions can be separated provided $\varepsilon$ is small since $\Lambda_j$ are mutually disjoint. Furthermore, if $c_1$ is a global minimum of $V(x)$, then Corollary \ref{Cor1.1} describes a multiple concentrating phenomenon.
\end{Rek}

\par
This paper is organized as follows. In section 2, we present some basic results which will be used later. In section 3, we prove the existence of ground state solutions for the modified problem for small $\varepsilon>0$. In section 4, we study the concentration phenomenon and complete the proof of Theorem \ref{Thm1.1}. Finally, we give some summaries and open problems in this direction.
\par
\vspace{3mm}
{\bf Notation.~}Throughout this paper, we make use of the following notations.
\begin{itemize}
\item[$\bullet$]  For any $R>0$ and for any $x\in \mathbb{R}^3$, $B_R(x)$ denotes the ball of radius $R$ centered at $x$;
\item[$\bullet$]  $\|\cdot\|_q$ denotes the usual norm of the space $L^q(\mathbb{R}^3),1\leq q\leq\infty$;
\item[$\bullet$]  $o_n(1)$ denotes $o_n(1)\rightarrow 0$ as $n\rightarrow\infty$;
\item[$\bullet$]  $C$ or $C_i(i=1,2,\cdots)$ are some positive constants may change from line to line.
\end{itemize}

\section{Variational setting and preliminary results}
Making the change of variable $x\mapsto \varepsilon x$, we can rewrite the equation \eqref{eq1.1} as the following equivalent form
\begin{equation}\label{eq2.1}
\begin{cases}
-\Delta u+V(\varepsilon x)u+K(\varepsilon x)\phi u=f(u)~&\text{in}~\mathbb{R}^3,\\
-\Delta \phi =K(\varepsilon x)u^2~&\text{in}~\mathbb{R}^3.
\end{cases}
\end{equation}
If $u$ is a solution of the equation \eqref{eq2.1}, then $v(x):=u(\frac{x}{\varepsilon})$ is a solution of the equation \eqref{eq1.1}. Thus, to study the equation \eqref{eq1.1}, it suffices to study the equation \eqref{eq2.1}.
In view of the presence of potential $V(x)$, we introduce the subspace
\begin{equation*}
H_\varepsilon=\bigg\{u\in H^1(\mathbb{R}^3):\int_{\mathbb{R}^3}V(\varepsilon x)u^2dx<\infty\bigg\},
\end{equation*}
which is a Hilbert space equipped with the inner product
\begin{equation*}
(u,v)_\varepsilon=\int_{\mathbb{R}^3}\nabla u\nabla vdx+\int_{\mathbb{R}^3}V(\varepsilon x)uvdx,
\end{equation*}
and the norm
\begin{equation*}
\|u\|_\varepsilon^2=\int_{\mathbb{R}^3}|\nabla u|^2dx+\int_{\mathbb{R}^3}V(\varepsilon x)u^2dx.
\end{equation*}
\par
As we know, system \eqref{eq2.1} is the Euler-Lagrange equation of the functional $J:H_\varepsilon\times \mathcal{D}^{1,2}(\mathbb{R}^3)\rightarrow\R$ defined by
\begin{equation*}
J(u,\phi)=\frac{1}{2}\|u\|_\varepsilon^2+\frac{1}{4}\int_{\mathbb{R}^3}|\nabla\phi|^2dx+\frac{1}{2}\int_{\mathbb{R}^3}K(x)\phi u^2dx-\int_{\mathbb{R}^3}F(u)dx.
\end{equation*}
\par
It is easy to see that $J$ exhibits a strong indefiniteness, namely it is unbounded both from below and from above on infinitely dimensional subspaces. This indefiniteness can be removed using the reduction method described in \cite{Benci-Fortunato1998TMNA}. Recall that by the Lax-Milgram theorem, we know that for every $u\in H^1(\mathbb{R}^3)$, there exists a unique $\phi_u\in \mathcal{D}^{1,2}(\mathbb{R}^3)$ such that
\begin{equation*}
-\Delta \phi_u=K(x)u^2
\end{equation*}
and $\phi_u$ can be expressed by
\begin{equation*}
\phi_u(x)=\frac{1}{4\pi}\frac{1}{|x|}*(K(x)u^2)=\frac{1}{4\pi}\int_{\mathbb{R}^3}\frac{K(y)u^2(y)}{|x-y|}dy,~\forall x\in \mathbb{R}^3,
\end{equation*}
which is called Riesz potential (see \cite{Landkof}), we will omit the constant $\pi$ in the following. It is clear that $\phi_u(x)\geq 0$ for all $x\in \mathbb{R}^3$. Then the system \eqref{eq2.1} can be reduced to the Schr\"{o}dinger equation with nonlocal term:
\begin{equation}\label{eq2.2}
-\Delta u+V(\varepsilon x)u+K(\varepsilon x)\phi_uu=f(u)~~\text{in}\ \mathbb{R}^3.
\end{equation}
\par
We define the energy functional $\Phi_\varepsilon$ corresponding to equation \eqref{eq2.2} by
\begin{equation*}
\Phi_\varepsilon(u)=\frac{1}{2}\int_{\mathbb{R}^3}|\nabla u|^2dx
+\frac{1}{2}\int_{\mathbb{R}^3}V(\varepsilon x)u^2dx+\frac{1}{4}\int_{\mathbb{R}^3}
\big[\frac{1}{|x|}*K(\varepsilon x)u^2\big]K(\varepsilon x)u^2dx-\int_{\mathbb{R}^3}F(u)dx.
\end{equation*}
It follows by standard arguments that $\Phi_\varepsilon\in C^1(H_\varepsilon,\mathbb{R})$. Also, for any $u,v\in H_\varepsilon$, one has
\begin{equation*}
\Phi_\varepsilon^\prime(u)v=\int_{\mathbb{R}^3}\nabla u \nabla vdx+\int_{\mathbb{R}^3}V(\varepsilon x)uvdx+\int_{\mathbb{R}^3}
\big[\frac{1}{|x|}*K(\varepsilon x)u^2\big]K(\varepsilon x)uvdx-\int_{\mathbb{R}^3}f(u)vdx.
\end{equation*}
Moreover, it is proved that critical points of $\Phi_\varepsilon$ are weak solutions of system \eqref{eq2.1}.
\par
Because we are concerned with the non-local problem in view of the presence of term $\phi_u$, we would like to recall the well-known Hardy-Littlewood-Sobolev inequality.
\begin{Lem}[Hardy-Littlewood-Sobolev inequality, \cite{Lieb-Loss2001book}]\label{Lem2.2}
Let $t,r>1$ and $0<\mu<3$ with
\begin{equation*}
\frac{1}{t}+\frac{\mu}{3}+\frac{1}{r}=2,
\end{equation*}
$f\in L^t(\R^3)$ and $h\in L^r(\R^3)$. There exists a sharp constant $C(t,\mu,r)$, independent of $f,h$ such that
\begin{equation*}
\int_{\R^3}\int_{\R^3}\frac{f(x)h(y)}{|x-y|^\mu}dydx\leq C(t,\mu,r)|f|_t|h|_r.
\end{equation*}
\end{Lem}
Using Hardy-Littlewood-Sobolev inequality, it is easy to check that
\begin{equation*}
\int_{\mathbb{R}^3}
\big[\frac{1}{|x|}*K(\varepsilon x)u^2\big]K(\varepsilon x)u^2dx\leq C\|u\|_{\frac{12}{5}}^4 \leq C\|u\|_\varepsilon^4.
\end{equation*}
To characteristic the ground state energy, we define the Nehari manifold by
\begin{equation*}
\mathcal{N}_\varepsilon=\{u\in H_\varepsilon\setminus\{0\}:\Phi_\varepsilon^\prime(u)u=0\}.
\end{equation*}
It is natural to define the ground state energy value by
\begin{equation*}
c_\varepsilon:=\inf_{u\in \mathcal{N}_\varepsilon}\Phi_\varepsilon (u).
\end{equation*}
If $c_\varepsilon$ is attained by $u\in \mathcal{N}_\varepsilon$, then $u$ is a critical point of $\Phi_\varepsilon$. Since $c_\varepsilon$ is the lowest level for $\Phi_\varepsilon$, $u$ is called a ground state solution of equation \eqref{eq2.1}.

\begin{Lem}\label{Lem2.5}
For any $u\in H_\varepsilon\setminus\{0\}$, there exists a unique $t_\varepsilon=t_\varepsilon(u)>0$ such that $t_\varepsilon u\in \mathcal{N}_\varepsilon$. Moreover,  $\Phi_\varepsilon(t_\varepsilon u)=\max\limits_{t\geq 0}\Phi_\varepsilon(tu)$.
\end{Lem}
\begin{proof}
For any $t>0$, let $h(t)=\Phi_\varepsilon(tu)$. It is easy to check that $h(0)=0,h(t)>0$ for $t>0$ small and $h(t)<0$ for $t$ large. Therefore, $\max\limits_{t\geq 0}h(t)$ is achieved at a $t_\varepsilon =t_\varepsilon(u)>0$ such that $h^\prime(t_\varepsilon)=0$ and $t_\varepsilon u\in \mathcal{N}_\varepsilon$. Since we have
\begin{equation*}
\begin{split}
h^\prime(t)=0&\Leftrightarrow tu\in \mathcal{N}_\varepsilon\\
&\Leftrightarrow t^2\|u\|_\varepsilon^2+t^4\int_{\mathbb{R}^3}
\big[\frac{1}{|x|}*K(\varepsilon x)u^2\big]K(\varepsilon x)u^2dx=\int_{\mathbb{R}^3}f(tu)tudx\\
&\Leftrightarrow \frac{\|u\|_\varepsilon^2}{t^2}+\int_{\mathbb{R}^3}
\big[\frac{1}{|x|}*K(\varepsilon x)u^2\big]K(\varepsilon x)u^2dx=\int_{\mathbb{R}^3}\frac{f(tu)u}{t^3}dx.
\end{split}
\end{equation*}
By the definition of $h$, the left hand side is a strictly decreasing function, while the right hand side is a strictly increasing function of $t>0$. Therefore, $\max\limits_{t\geq 0}h(t)$ is achieved at a unique $t_\varepsilon=t_\varepsilon(u)>0$ such that $h^\prime (t_\varepsilon)=0$ and $t_\varepsilon u\in \mathcal{N}_\varepsilon$. Hence, we complete the proof.
\end{proof}

In order to obtain a ground state solution, we need a characterization of the ground state energy. Following \cite{Willem1996book}, we define
\begin{equation*}
c_\varepsilon=\inf_{u\in \mathcal{N}_\varepsilon}\Phi_\varepsilon(u),\quad c_\varepsilon^*=\inf_{\gamma\in \Gamma_\varepsilon}\max_{t\in[0,1]}\Phi_\varepsilon(\gamma(t)), \quad c_\varepsilon^{**}=\inf_{u\in H_\varepsilon\setminus\{0\}}\max_{t\geq 0}\Phi_\varepsilon(tu),
\end{equation*}
where $\Gamma_\varepsilon=\{\gamma\in C([0,1],H_\varepsilon):\gamma(0)=0,\Phi_\varepsilon(\gamma(1))<0\}$.

\begin{Lem}\label{Lem2.6}
$c_\varepsilon=c_\varepsilon^*=c_\varepsilon^{**}>0$.
\end{Lem}
\begin{proof}
It follows from Lemma \ref{Lem2.5} that $c_\varepsilon=c_\varepsilon^{**}$. Note that for any $u\in H_\varepsilon\setminus\{0\}$, there exists some $t_0>0$ such that $\Phi_\varepsilon(t_0u)<0$. Define a path $\gamma:[0,1]\rightarrow H_\varepsilon$ by $\gamma(t)=tt_0u$. Clearly, $\gamma\in \Gamma_\varepsilon$ and consequently, $c_\varepsilon^*\leq c_\varepsilon^{**}$. Next we prove that $c_\varepsilon\leq c_\varepsilon^*$. We have to show that given $\gamma\in \Gamma_\varepsilon$, there exist $\tilde{t}\in[0,1]$ such that $\gamma(t)\in \mathcal{N}_\varepsilon$. Assuming the contrary we have $\gamma(t)\notin \mathcal{N}_\varepsilon$ for all $t\in[0,1]$. In view of $(f_1)$, we have
\begin{equation*}
\int_{\mathbb{R}^3}|\nabla \gamma(t)|^2dx
+\int_{\mathbb{R}^3}V(\varepsilon x)\gamma^2(t)dx+\int_{\mathbb{R}^3}
\big[\frac{1}{|x|}*K(\varepsilon x)\gamma^2(t)\big]K(\varepsilon x)\gamma^2(t)dx>\int_{\mathbb{R}^3}f(\gamma(t))\gamma(t)dx
\end{equation*}
which implies that
\begin{equation*}
\Phi_\varepsilon(\gamma(t))>\frac{1}{4}\int_{\mathbb{R}^3}|\nabla \gamma(t)|^2dx
+\frac{1}{4}\int_{\mathbb{R}^3}V(\varepsilon x)\gamma^2(t)dx+\int_{\mathbb{R}^3}\bigg[\frac{1}{4}f(\gamma(t))\gamma(t)-F(\gamma(t))\bigg]dx\geq 0
\end{equation*}
which is a contradiction with the definition of $\gamma$.
\end{proof}
\par
In order to prove our main result, we will make use of the autonomous problem. Precisely, for any $\mu>0,\nu\geq0$, we consider the following constant coefficient system
\begin{equation}\label{eq2.3}
\begin{cases}
-\Delta u+\mu u+\nu\phi u=f(u)&\text{in}\ \mathbb{R}^3,\\
-\Delta \phi=\nu u^2&\text{in}\ \mathbb{R}^3,
\end{cases}
\end{equation}
and the corresponding energy functional
\begin{equation*}
\mathcal{J}_{\mu\nu}(u)=
\frac{1}{2}\int_{\mathbb{R}^3}|\nabla u|^2dx+\frac{\mu}{2}\int_{\mathbb{R}^3}u^2dx+\frac{\nu^2}{4}\int_{\mathbb{R}^3}\big[\frac{1}{|x|}*u^2\big]u^2dx
-\int_{\mathbb{R}^3}F(u)dx,
\end{equation*}
for $u\in H^1(\mathbb{R}^3)$. Define the Nehari manifold associated with \eqref{eq2.3} by
\begin{equation*}
\mathcal{N}^{\mu\nu}:=\big\{u\in H^1(\mathbb{R}^3)\setminus \{0\}:\mathcal{J}_{\mu\nu}^\prime (u)u=0\big\},
\end{equation*}
and the ground state energy
\begin{equation*}
c_{\mu\nu}=\inf_{u\in \mathcal{N}^{\mu\nu}}\mathcal{J}_{\mu\nu}(u).
\end{equation*}
The energy value $c_{\mu\nu}$ and the manifold $\mathcal{N}^{\mu\nu}$ have properties similar to those of $c_\varepsilon$ and $\mathcal{N}_\varepsilon$ stated in Lemmas \ref{Lem2.5}-\ref{Lem2.6}. Hence, for each $u\in H^1(\mathbb{R}^3)\setminus \{0\}$, there exists a unique $t_u>0$ such that $t_uu\in \mathcal{N}^{\mu\nu}$. Note that, $c_{\mu\nu}$ is attained at some positive $u\in H^1(\mathbb{R}^3)$, see \cite{He2011zamp} for $\nu>0$ and \cite{Rabinowitz1992} for $\nu=0$.

\par
The following Lemma describes a comparison between the ground state energy values for different parameters $\mu,\nu$, which will play an important role in proving the existence result in the following .
\begin{Lem}\label{Lem2.7}
Let $\mu_j>0$ and $\nu_j\geq0$ for $j=1,2$ with $\mu_1\leq \mu_2$ and $\nu_1\leq \nu_2$. Then $c_{\mu_1\nu_1}\leq c_{\mu_2\nu_2}$. In particular, if one of inequalities is strict, then $c_{\mu_1\nu_1}<c_{\mu_2\nu_2}$.
\end{Lem}
\begin{proof}
Let $u\in \mathcal{N}^{\mu_2\nu_2}$ be such that
 \begin{equation*}
c_{\mu_2\nu_2}=\mathcal{J}_{\mu_2\nu_2}(u)=\max_{t>0}\mathcal{J}_{\mu_2\nu_2}(tu).
\end{equation*}
Let $u_0=t_1u$ be such that $\mathcal{J}_{\mu_1\nu_1}(u_0)=\max\limits_{t>0}\mathcal{J}_{\mu_1\nu_1}(tu)$. One has
 \begin{equation*}
 \begin{split}
c_{\mu_2\nu_2}&=\mathcal{J}_{\mu_2\nu_2}(u)\geq \mathcal{J}_{\mu_2\nu_2}(u_0)\\
&=\mathcal{J}_{\mu_1\nu_1}(u_0)+\frac{\nu_2^2-\nu_1^2}{4}\int_{\mathbb{R}^3}
\big[\frac{1}{|x|}*u_0^2\big]u_0^2dx+\frac{\mu_2-\mu_1}{2}\int_{\mathbb{R}^3}u_0^2dx\\
&\geq c_{\mu_1\nu_1}.
\end{split}
\end{equation*}
Thus, we complete the proof.
\end{proof}

\section{The modified problem}
In what follows, we will not work directly with the functional $\Phi_\varepsilon$, because we have some difficulties to verify the $(PS)$ condition. We will adapt for our case an argument explored by the penalization method introduction by del Pino and Felmer \cite{Pino1996CVPDE}, and build a convenient modification of the energy functional $\Phi_\varepsilon$ such that it satisfies the $(PS)$ condition.
\par
Note that, by $(f_1)$ and $(f_4)$, it is easy to check that
\begin{equation}\label{eq3.1}
f(t)=o(t)\quad\text{as}\quad t\rightarrow 0
\end{equation}
and
\begin{equation}\label{eq3.2}
\text{The~function}~\frac{f(t)}{t}~\text{is~increasing~for}~t>0.
\end{equation}
Now, let us fix $\kappa>1$ and then by \eqref{eq3.1}-\eqref{eq3.2}, there exists unique $a>0$ such that
\begin{equation*}
\frac{f(a)}{a}=\frac{V_0}{\kappa}
\end{equation*}
 where $V_0 > 0$ was given in $(V_2)$. Using these numbers, we set the functions
\begin{equation*}
f_*(t)=
\begin{cases}
f(t),&\text{if}~t\leq a\\
\frac{V_0}{\kappa}t,&\text{if}~t>a
\end{cases}
\end{equation*}
and
\begin{equation*}
g(x,t)=\chi(x)f(t)+(1-\chi(x))f_*(t),
\end{equation*}
where $\Lambda$ is given in $(V_2)$ and $\chi$ is the characteristic function of the set $\Lambda$ in $\mathbb{R}^3$.
Clearly, $f_*\in C(\mathbb{R}^+,\mathbb{R}^+)$ and $f_*(t)\leq f(t),f_*(t)\leq \frac{V_0}{\kappa}t$ for all $t\geq 0$. From hypotheses $(f_1)$-$(f_4)$, it is easy to check that $g$ is a Carath$\acute{e}$odory function and satisfies the following properties.
\begin{itemize}
\item[$(g_1)$] $0<\mu G(x,t)\leq g(x,t)t$~for all $(x,t)\in \Lambda \times (0,\infty)$;
\item[$(g_2)$] $0\leq 2G(x,t)\leq g(x,t)t\leq \frac{V_0}{\kappa}t^2$~for all $x\in \Lambda^c\times (0,\infty)$,
\end{itemize}
where $G(x,t)=\int_0^tg(x,s)ds$. From $(g_1)$-$(g_2)$, it is easy to check that
\begin{itemize}
\item[$(g_3)$] $\frac{1}{4\kappa}V(x)t^2+\frac{1}{4}g(x,t)t-G(x,t)\geq 0$~for all $(x,t)\in \mathbb{R}^3\times \mathbb{R}$;
\item[$(g_4)$] $M(x,t):=V(x)t^2-g(x,t)t\geq (1-\frac{1}{\kappa})V(x)t^2\geq 0$~for all $(x,t)\in \Lambda^c\times \mathbb{R}$.
\end{itemize}

Now we study the modified problem
\begin{equation}\label{eq3.3}
\begin{cases}
-\Delta u+V(\varepsilon x)u+K(\varepsilon x)\phi u=g(\varepsilon x,u)~&\text{in}~\mathbb{R}^3,\\
-\Delta \phi =K(\varepsilon x)u^2~&\text{in}~\mathbb{R}^3.
\end{cases}
\end{equation}
Note that if $u_\varepsilon$ is a solution of \eqref{eq3.3} with
\begin{equation*}
u_\varepsilon(x)\leq a~,\forall x\in \Lambda_\varepsilon^c,
\end{equation*}
then $u_\varepsilon$ is a solution of system \eqref{eq2.1}, where $\Lambda_\varepsilon:=\{x\in \mathbb{R}^3:\varepsilon x\in \Lambda\}$.
\par
Now, we define the modified functional $\tilde{\Phi}_\varepsilon: H_\varepsilon\rightarrow \mathbb{R}$ as
\begin{equation*}
\tilde{\Phi}_\varepsilon(u)=\frac{1}{2}\int_{\mathbb{R}^3}|\nabla u|^2dx
+\frac{1}{2}\int_{\mathbb{R}^3}V(\varepsilon x)u^2dx+\frac{1}{4}\int_{\mathbb{R}^3}
\big[\frac{1}{|x|}*K(\varepsilon x)u^2\big]K(\varepsilon x)u^2dx-\int_{\mathbb{R}^3}G(\varepsilon x,u)dx,
\end{equation*}
which is of class $C^1$ on $H_\varepsilon$ and its critical points are the solutions of \eqref{eq3.3}.

The following Lemma implies that the functional $\tilde{\Phi}_\varepsilon$ possesses the Mountain Pass structure.
\begin{Lem}\label{Lem3.1}
The functional $\tilde{\Phi}_\varepsilon$ satisfies the following conditions:
\begin{itemize}
\item[$(i)$]  There exists $\alpha,\rho>0$ such that $\tilde{\Phi}_\varepsilon(u)\geq\alpha$ with $\|u\|_\varepsilon=\rho$;
\item[$(ii)$] There exists $e\in H_\varepsilon$ satisfying $\|e\|_\varepsilon>\rho$ such that $\tilde{\Phi}_\varepsilon(e)<0$.
\end{itemize}
\end{Lem}
\begin{proof}
$(i).$ For any $u\in H_\varepsilon\backslash\{0\}$, by the Sobolev inequality, we have
\begin{equation*}
\begin{split}
\tilde{\Phi}_\varepsilon(u)&=\frac{1}{2}\|u\|_\varepsilon^2+\frac{1}{4}\int_{\mathbb{R}^3}
\big[\frac{1}{|x|}*K(\varepsilon x)u^2\big]K(\varepsilon x)u^2dx-\int_{\mathbb{R}^3}G(\varepsilon x,u)dx\\
&\geq \frac{1}{2}\|u\|_\varepsilon^2-\int_{\mathbb{R}^3}F(u)dx\\
&\geq \frac{1}{4}\|u\|_\varepsilon^2-C\|u\|_\varepsilon^p.
\end{split}
\end{equation*}
Since $p>4$, hence, we can choose some $\rho>0$ such that
\begin{equation*}
\tilde{\Phi}_\varepsilon(u)\geq \alpha~\text{with}~\|u\|_\varepsilon=\rho.
\end{equation*}
\par
$(ii).$ Fix $\varphi\in H_\varepsilon\setminus\{0\}$ with $\text{supp}\varphi\subset \Lambda_\varepsilon$. Then, for $t>0$, we have
\begin{equation*}
\begin{split}
\tilde{\Phi}_\varepsilon(t \varphi)&\leq \frac{t^2}{2}\|\varphi\|_\varepsilon^2+\frac{t^4}{4}\|K\|_\infty^2\int_{\mathbb{R}^3}
\big[\frac{1}{|x|}*\varphi^2\big]\varphi^2dx-\int_{\Lambda_\varepsilon}F(t \varphi)dx\\
&\leq\frac{t^2}{2}\|\varphi\|_\varepsilon^2+\frac{t^4}{4}\|K\|_\infty^2\int_{\mathbb{R}^3}
\big[\frac{1}{|x|}*\varphi^2\big]\varphi^2dx-C_1t^\mu\int_{\Lambda_\varepsilon}|\varphi|^\mu dx+C_2|\Lambda_\varepsilon|,
\end{split}
\end{equation*}
where we have used the Remark \ref{Rek1.1}. Therefore, $(ii)$ follows with $e=t\varphi$ and $t$ large enough.
\end{proof}
\par
It follows from Lemma \ref{Lem3.1} and the Mountain Pass theorem without (PS) condition \cite{Willem1996book} that there exist a $(P S)_{c}$ sequence $\left\{u_{n}\right\} \subset H_{\varepsilon}$ such that
\begin{equation}\label{eq3.4}
\tilde{\Phi}_\varepsilon\left(u_{n}\right) \rightarrow \tilde{c}_{\varepsilon}\quad\text{and}\quad \tilde{\Phi}_\varepsilon^{\prime}\left(u_{n}\right) \rightarrow 0\quad \text{in}~H_{\varepsilon}^{-1},
\end{equation}
with the mountain pass level
$$\tilde{c}_{\varepsilon}=\inf _{\gamma \in \tilde{\Gamma}_\varepsilon} \sup _{t \in[0,1]} \tilde{\Phi}_\varepsilon(\gamma(t))>0,$$
where $\tilde{\Gamma}_\varepsilon=\left\{\gamma \in C\left([0,1], H_{\varepsilon}\right): \gamma(0)=0, \tilde{\Phi}_\varepsilon(\gamma(1))<0\right\} .$ Moreover, we have the following Lemmas.

\begin{Lem}\label{Lem3.2}
$\{u_n\}$ is bounded in $H_\varepsilon$.
\end{Lem}
\begin{proof}
Let $\{u_n\}\subset H_\varepsilon$ be a $(PS)_c$ sequence for $\tilde{\Phi}_\varepsilon$, that is
\begin{equation*}
\tilde{\Phi}_\varepsilon(u_n)\rightarrow \tilde{c}_\varepsilon\quad\text{and}\quad\tilde{\Phi}_\varepsilon^\prime(u_n)\rightarrow 0.
\end{equation*}
By $(g_4)$, for $n$ large enough, one has
\begin{equation*}
\begin{split}
C+C\|u_n\|_\varepsilon&\geq \tilde{\Phi}_\varepsilon(u_n)-\frac{1}{4}\tilde{\Phi}_\varepsilon^\prime(u_n)u_n\\
&=\frac{1}{4}\|u_n\|_\varepsilon^2+\int_{\mathbb{R}^3}\big[\frac{1}{4}g(\varepsilon x,u_n)-G(\varepsilon x,u_n)\big]dx\\
&\geq\frac{1}{4}(1-\frac{1}{\kappa})\|u_n\|_\varepsilon^2+\int_{\mathbb{R}^3}\big[\frac{1}{4\kappa}V(\varepsilon x)u_n^2+\frac{1}{4}g(\varepsilon x,u_n)-G(\varepsilon x,u_n)\big]dx\\
&\geq \frac{1}{4}(1-\frac{1}{\kappa})\|u_n\|_\varepsilon^2
\end{split}
\end{equation*}
which implies that $\{u_n\}$ is bounded in $H_\varepsilon$.
\end{proof}
\par

Adopting similar arguments as in Lemma \ref{Lem2.6} we have the following equivalent characterization of $\tilde{c}_{\varepsilon}$.

\begin{Lem}\label{Lem3.3}
$\tilde{c}_\varepsilon=\inf\limits_{u\in H_\varepsilon\setminus\{0\}}\max\limits_{t\geq 0}\tilde{\Phi}_\varepsilon(tu)$ for any $\varepsilon>0$ small.
\end{Lem}
\begin{proof}
For any $u\in H_\varepsilon\setminus\{0\}$, we first define a function
\begin{equation*}
\begin{split}
\mathcal{J}_\infty(tu):&=\frac{t^2}{2}\big(\int_{\mathbb{R}^3}|\nabla u|^2dx+\|V\|_\infty\int_{\mathbb{R}^3}u^2dx\big)+\frac{t^4}{4}\|K\|_\infty^2\int_{\mathbb{R}^3}
\big[\frac{1}{|x|}*u^2\big]u^2dx-\int_{\mathbb{R}^3}F(tu)dx.
\end{split}
\end{equation*}
Thus, we have
\begin{equation*}
\begin{split}
\mathcal{J}_\infty(tu)\rightarrow -\infty~\text{as}~t\rightarrow \infty,
\end{split}
\end{equation*}
which implies that there exists $t_0>0$ large enough such that
\begin{equation*}
\mathcal{J}_\infty(t_0u)<-2.
\end{equation*}
Hence there is $R_0>0$ such that
\begin{equation*}
\frac{t_0^2}{2}\big(\int_{\mathbb{R}^3}|\nabla u|^2dx+\|V\|_\infty\int_{\mathbb{R}^3}u^2dx\big)+\frac{t_0^4}{4}\|K\|_\infty^2\int_{\mathbb{R}^3}
\big[\frac{1}{|x|}*u^2\big]u^2dx-\int_{B_{R_0}}F(t_0u)dx<-1.
\end{equation*}
Therefore,
\begin{equation*}
\begin{split}
\tilde{\Phi}_\varepsilon(t_0u)&=\frac{t^2}{2}\|u\|_\varepsilon^2+\frac{t_0^4}{4}\int_{\mathbb{R}^3}
\big[\frac{1}{|x|}*K(\varepsilon x)u^2\big]K(\varepsilon x)u^2dx-\int_{\mathbb{R}^3}G(\varepsilon x,t_0u)dx\\
&\leq\frac{t_0^2}{2}\|u\|_\varepsilon^2+\frac{t_0^4}{4}\|K\|_\infty^2\int_{\mathbb{R}^3}
\big[\frac{1}{|x|}*u^2\big]u^2dx-\int_{\Lambda_\varepsilon}F(t_0u)dx\\
&\leq\frac{t_0^2}{2}\big(\int_{\mathbb{R}^3}|\nabla u|^2dx+\|V\|_\infty\int_{\mathbb{R}^3}u^2dx\big)+\frac{t_0^4}{4}\|K\|_\infty^2\int_{\mathbb{R}^3}
\big[\frac{1}{|x|}*u^2\big]u^2dx-\int_{B_{R_0}}F(t_0u)dx\\
&<-1.
\end{split}
\end{equation*}
Thus, we have
\begin{equation*}
\tilde{c}_\varepsilon\leq\inf_{u\in H_\varepsilon\setminus\{0\}}\max_{t\geq 0}\tilde{\Phi}_\varepsilon(tu)
\end{equation*}
On the other hand, it is easy to see that $t\mapsto \tilde{\Phi}_\varepsilon(tu)$ has at most one nontrivial critical point $t=t(u)>0$. As in \cite[Lemma 1.2]{Pino1996CVPDE}, we define
\begin{equation*}
\mathcal{ M}_\varepsilon:=\{t(u)u:u\in H_\varepsilon\setminus \{0\},t(u)<\infty\}.
\end{equation*}
Then
\begin{equation*}
\inf_{u\in H_\varepsilon\setminus\{0\}}\max_{t\geq 0}\tilde{\Phi}_\varepsilon(tu)=\inf_{v\in \mathcal{M}_\varepsilon}\tilde{\Phi}_\varepsilon(v).
\end{equation*}
Thus we only need to show that given $\gamma\in \tilde{\Gamma}_\varepsilon$, there exists $\tilde{t}\in [0,1]$ such that $\gamma(\tilde{t})\in \mathcal{M}_\varepsilon$. Similar to the proof in Lemma \ref{Lem2.6}, this completes the proof.
\end{proof}

\begin{Lem} \label{Lem3.4}
There exists $\varepsilon^*>0$ such that, for all $\varepsilon\in (0,\varepsilon^*)$, there exist $\{y_\varepsilon\}\subset \mathbb{R}^3$ and $R$, $\delta>0$ such that
\begin{equation*}
\int_{B_R(y_\varepsilon)}u^2dx\geq \delta.
\end{equation*}
\end{Lem}
\begin{proof}
Assume by contradiction that there exists a sequence $\varepsilon_j\rightarrow 0$ as $j\rightarrow \infty$, such that for any $R>0$,
\begin{equation*}
\lim_{j\rightarrow\infty}\sup_{y\in \mathbb{R}^3}\int_{B_R(y)}u_{\varepsilon_j}^2dx=0.
\end{equation*}
Thus, by Lion's concentration principle \cite[Lemma 1.1]{lions1}, we have
\begin{equation*}
u_{\varepsilon_j}\rightarrow0\ \text{in}\ L^r(\mathbb{R}^3)\ \text{for}\ 2<r<6.
\end{equation*}
Thus, by Remark \ref{Rek1.1}
 \begin{equation*}
\int_{\mathbb{R}^3}g(\varepsilon_j x,u_{\varepsilon_j})u_{\varepsilon_j}dx\rightarrow 0~\hbox{as}~j\rightarrow\infty,
\end{equation*}
and by Lemma \ref{Lem2.2}
 \begin{equation*}
\int_{\mathbb{R}^3}
\big[\frac{1}{|x|}*K(\varepsilon_j x)u_{\varepsilon_j}^2\big]K(\varepsilon_j x)u_{\varepsilon_j}^2dx\rightarrow 0~\hbox{as}~j\rightarrow\infty.
\end{equation*}
Therefore,
\begin{equation}\label{eq3.6}
\|u_{\varepsilon_j}\|_{\varepsilon_j}^2=-\int_{\mathbb{R}^3}\big[\frac{1}{|x|}*K(\varepsilon_j x)u_{\varepsilon_j}^2\big]K(\varepsilon_j x)u_{\varepsilon_j}^2dx
-\int_{\mathbb{R}^3}g(\varepsilon_j x,u_{\varepsilon_j})u_{\varepsilon_j}dx\rightarrow0.
\end{equation}
Moreover, it is not difficulty to check that exists constant $r^*$ such that
\begin{equation*}
\|u_{\varepsilon_j}\|_{\varepsilon_j} \geq r^{*}>0,
\end{equation*}
which is a contradiction with \eqref{eq3.6}.
\end{proof}

\begin{Lem}\label{Lem3.5}
The functional $\tilde{\Phi}_\varepsilon$ satisfies the $(PS)_c$ condition.
\end{Lem}
\begin{proof}
Let $\{u_n\}\subset H_\varepsilon$ be a $(PS)_c$ sequence for $\tilde{\Phi}_\varepsilon$, it follows from Lemmas \ref{Lem3.2}-\ref{Lem3.4} that $\{u_n\}$ is bounded in $H_\varepsilon$ and we can assume that there is $u\in H_\varepsilon$ such that
\begin{equation}\label{eq3.5}
\begin{aligned}
&u_n\rightharpoonup u\quad\hbox{in}~H_\varepsilon\\
&u_n(x)\to u(x)\quad\hbox{a.e. in}\  \mathbb{R}^3,\\
&u_n\rightarrow u
\quad\hbox{in}\  L^{q}_{loc}(\mathbb{R}^3), \ \hbox{for }\ 1\leq q<6
\end{aligned}
\end{equation}
and $\tilde{\Phi}_\varepsilon^\prime(u)=0$. Thus,
\begin{equation*}
\int_{\mathbb{R}^3}|\nabla u|^2dx+\int_{\Lambda_\varepsilon}V(\varepsilon x)u^2dx+\int_{\mathbb{R}^3}
\big[\frac{1}{|x|}*K(\varepsilon x)u^2\big]K(\varepsilon x)u^2dx+
\int_{\Lambda_\varepsilon^c}M(\varepsilon x,u)dx=\int_{\Lambda_\varepsilon}f(u)udx.
\end{equation*}
On the other hand, using the limit $\Phi_\varepsilon^\prime(u_n)u_n=o_n(1)$, we deduce that
\begin{equation*}
\int_{\mathbb{R}^3}|\nabla u_n|^2dx+\int_{\Lambda_\varepsilon}V(\varepsilon x)u_n^2dx+\int_{\mathbb{R}^3}
\big[\frac{1}{|x|}*K(\varepsilon x)u_n^2\big]K(\varepsilon x)u_n^2dx+
\int_{\Lambda_\varepsilon^c}M(x,u_n)dx=\int_{\Lambda_\varepsilon}f(u_n)u_ndx+o_n(1).
\end{equation*}
Since $\Lambda_\varepsilon$ is bounded, the compactness of the Sobolev embedding gives
\begin{equation*}
\lim_{n\rightarrow \infty}\int_{\Lambda_\varepsilon}f(u_n)u_ndx=\int_{\Lambda_\varepsilon}f(u)udx
\end{equation*}
and
\begin{equation}\label{eq3.7}
\lim_{n\rightarrow \infty}\int_{\Lambda_\varepsilon}V(\varepsilon x)u_n^2dx=\int_{\Lambda_\varepsilon}V(\varepsilon x)u^2dx.
\end{equation}
Therefore,
\begin{equation*}
\begin{split}
&\limsup_{n\rightarrow\infty}\int_{\mathbb{R}^3}|\nabla u_n|^2dx+\int_{\mathbb{R}^3}
\big[\frac{1}{|x|}*K(\varepsilon x)u_n^2\big]K(\varepsilon x)u_n^2dx+
\int_{\Lambda_\varepsilon^c}M(x,u_n)dx\\
&=\int_{\mathbb{R}^3}|\nabla u|^2dx+\int_{\mathbb{R}^3}
\big[\frac{1}{|x|}*K(\varepsilon x)u^2\big]K(\varepsilon x)u^2dx+
\int_{\Lambda_\varepsilon^c}M(x,u)dx.
\end{split}
\end{equation*}
Now, recalling that $M(x,t)\geq 0$, the Fatou's lemma lead to
\begin{equation*}
\begin{split}
&\liminf_{n\rightarrow\infty}\int_{\mathbb{R}^3}|\nabla u_n|^2dx+\int_{\mathbb{R}^3}
\big[\frac{1}{|x|}*K(\varepsilon x)u_n^2\big]K(\varepsilon x)u_n^2dx+
\int_{\Lambda_\varepsilon^c}M(x,u_n)dx\\
&\geq\int_{\mathbb{R}^3}|\nabla u|^2dx+\int_{\mathbb{R}^3}
\big[\frac{1}{|x|}*K(\varepsilon x)u^2\big]K(\varepsilon x)u^2dx+
\int_{\Lambda_\varepsilon^c}M(x,u)dx.
\end{split}
\end{equation*}
Hence,
\begin{equation*}
\begin{split}
&\int_{\mathbb{R}^3}|\nabla u|^2dx\leq\liminf_{n\rightarrow\infty}\int_{\mathbb{R}^3}|\nabla u_n|^2dx\leq \limsup_{n\rightarrow\infty}\int_{\mathbb{R}^3}|\nabla u_n|^2dx\\
&\leq \limsup_{n\rightarrow\infty}\int_{\mathbb{R}^3}|\nabla u_n|^2dx+\liminf_{n\rightarrow\infty}\int_{\Lambda_\varepsilon^c}M(x, u_n)dx-\int_{\Lambda_\varepsilon^c}M(x,u)dx\\
&\quad+\liminf_{n\rightarrow\infty}\int_{\mathbb{R}^3}
\big[\frac{1}{|x|}*K(\varepsilon x)u_n^2\big]K(\varepsilon x)u_n^2dx-\int_{\mathbb{R}^3}
\big[\frac{1}{|x|}*K(\varepsilon x)u^2\big]K(\varepsilon x)u^2dx\\
&\leq\limsup_{n\rightarrow\infty}\bigg(\int_{\mathbb{R}^3}|\nabla u_n|^2dx+\int_{\Lambda_\varepsilon^c}M(x, u_n)dx+\int_{\mathbb{R}^3}
\big[\frac{1}{|x|}*K(\varepsilon x)u_n^2\big]K(\varepsilon x)u_n^2dx\bigg)\\
&\quad-\int_{\Lambda_\varepsilon^c}M(x,u)dx-\int_{\mathbb{R}^3}
\big[\frac{1}{|x|}*K(\varepsilon x)u^2\big]K(\varepsilon x)u^2dx\\
&=\int_{\mathbb{R}^3}|\nabla u|^2dx
\end{split}
\end{equation*}
which implies that
\begin{equation*}
\lim_{n\rightarrow\infty}\int_{\mathbb{R}^3}|\nabla u_n|^2dx=\int_{\mathbb{R}^3}|\nabla u|^2dx.
\end{equation*}
Similarly,
\begin{equation*}
\lim_{n\rightarrow\infty}\int_{\Lambda_\varepsilon^c}M(x,u_n)dx=\int_{\Lambda_\varepsilon^c}M(x,u)dx
\end{equation*}
and
\begin{equation*}
\lim_{n\rightarrow\infty}\int_{\mathbb{R}^3}
\big[\frac{1}{|x|}*K(\varepsilon x)u_n^2\big]K(\varepsilon x)u_n^2dx=\int_{\mathbb{R}^3}
\big[\frac{1}{|x|}*K(\varepsilon x)u^2\big]K(\varepsilon x)u^2dx.
\end{equation*}
The last limit combined with the fact
\begin{equation*}
M(x,t)\leq V(x)t^2\leq \frac{\kappa}{\kappa-1}M(x,t)~\text{for~any}~(x,t)\in \Lambda^c\times \mathbb{R},
\end{equation*}
yields
\begin{equation}\label{eq3.8}
\lim_{n\rightarrow\infty}\int_{\Lambda_\varepsilon^c}V(\varepsilon x)u_n^2dx=\int_{\Lambda_\varepsilon^c}V(\varepsilon x)u^2dx.
\end{equation}
Hence, \eqref{eq3.7} and \eqref{eq3.8} imply that
\begin{equation*}
\lim_{n\rightarrow\infty}\int_{\mathbb{R}^3}V(\varepsilon x)u_n^2dx=\int_{\mathbb{R}^3}V(\varepsilon x)u^2dx.
\end{equation*}
Thus,
\begin{equation*}
\lim_{n\rightarrow\infty}\big(\int_{\mathbb{R}^3}|\nabla v_n|^2dx+\int_{\mathbb{R}^3}v_n^2dx\big)=\int_{\mathbb{R}^3}| \nabla u|^2dx+\int_{\mathbb{R}^3}u^2dx.
\end{equation*}
Together with $v_n\rightharpoonup u$ in $H^1(\mathbb{R}^3)$, we have $v_n\rightarrow u$ in $H^1(\mathbb{R}^3)$.
\end{proof}
\par

\begin{Lem}\label{Lem3.6}
The functional $\tilde{\Phi}_\varepsilon$ has a positive critical point $u_\varepsilon$ such that
\begin{equation*}
\tilde{\Phi}_\varepsilon(u_\varepsilon)=\tilde{c}_\varepsilon~~\text{and}~~\tilde{\Phi}_\varepsilon^\prime(u_\varepsilon)=0.
\end{equation*}
\end{Lem}
\begin{proof}
From Lemma \ref{Lem3.1} and \eqref{eq3.4}, we know that there is a nontrivial $u_\varepsilon\in H_\varepsilon$ such that
\begin{equation*}
\tilde{\Phi}_\varepsilon(u_\varepsilon)=\tilde{c}_\varepsilon~~\text{and}~~\tilde{\Phi}_\varepsilon^\prime(u_\varepsilon)=0.
\end{equation*}
Moreover, the function $u_\varepsilon$ is nonnegative, because
\begin{equation*}
\tilde{\Phi}_\varepsilon^\prime(u_\varepsilon)u_\varepsilon^-=0\Rightarrow u_\varepsilon^-=0,
\end{equation*}
where $u_\varepsilon^-=\min \{u_\varepsilon(x),0\}$. Thus, by strong maximal principle, we get a positive critical point.
\end{proof}
\par

To find a new upper bound for $\tilde{c}_\varepsilon$, we need to investigate the "limit" problem:
\begin{equation*}
-\Delta u+V_0u=f(u),\quad \text{in}~\mathbb{R}^3.
\end{equation*}
Firstly, denote the standard norm of $H_{V_0}$ by
$$
\|u\|_{V_0}:=\left(\int_{\mathbb{R}^3}|\nabla u|^{2}+V_0u^2 d x\right)^{1 / 2}
$$
The associated energy functional is
\begin{equation*}
\mathcal{J}_{V_0}(u)=
\frac{1}{2}\int_{\mathbb{R}^3}|\nabla u|^2dx+\frac{V_0}{2}\int_{\mathbb{R}^3}u^2dx-\int_{\mathbb{R}^3}F(u)dx.
\end{equation*}
It is easy to check that $\mathcal{J}_{V_0}$ is well defined on $H_{V_0}$ and
$\mathcal{J}_{V_0}\in C^{1}\left(H_{V_0}, \mathbb{R}\right) .$ Then we can define
$$
\mathcal{N}^{V_0}=\left\{u \in H_{V_0} \backslash\{0\} \mid\left\langle \mathcal{J}_{V_0}^{\prime}(u), u\right\rangle=0\right\}
$$
and
$$
c_{V_0}=\inf _{u \in \mathcal{N}^{V_0}} \mathcal{J}_{V_0}(u).
$$

\begin{Rek}
For $c_{V_0}, \mathcal{J}_{V_0}$ and $\mathcal{N}^{V_0}$, there are similar results obtained from Lemma \ref{Lem3.1} to Lemma \ref{Lem3.6}. By Mountain pass Theorem, we can see that there exists $u \in H_{V_0}$ such that $\mathcal{J}_{V_0}(u)=c_{V_0}$ and $\mathcal{J}_{V_0}^{\prime}(u)=0$.
\end{Rek}
\par
The next lemma establishes an important relation between $\tilde{c}_\varepsilon$ and $c_{V_0}$.
\begin{Lem}\label{Lem3.7}
$\limsup\limits_{\varepsilon\rightarrow 0}\tilde{c}_\varepsilon\leq c_{V_0}$.
\end{Lem}
\begin{proof}
Let $w$ be a positive ground state solution of $\mathcal{J}_{V_0}$, that is, $w\in \mathcal{N}^{V_0}, \mathcal{J}_{V_0}(w)=c_{V_0}$ and $\mathcal{J}_{V_0}^\prime(w)=0$. Set
$w_\varepsilon(x)=\eta(\varepsilon x)w(x)$, where $\eta$ is a smooth function compactly supported in $\Lambda$, such that $\eta=1$ in a small neighborhood of origin in $\Lambda$. From the definition of $g$, one has
\begin{equation*}
\int_{\mathbb{R}^3}g(\varepsilon x,w_\varepsilon)w_\varepsilon dx=\int_{\mathbb{R}^3}f(w_\varepsilon)w_\varepsilon dx,~\text{and}~\int_{\mathbb{R}^3}G(\varepsilon x,w_\varepsilon)dx=\int_{\mathbb{R}^3}F(w_\varepsilon)dx.
\end{equation*}
Thus, Lemma \ref{Lem2.5} implies that there exists $t_\varepsilon>0$ such that $t_\varepsilon w_\varepsilon \in \mathcal{N}_\varepsilon$ and satisfying
\begin{equation}\label{eq3.9}
\tilde{\Phi}_\varepsilon(t_\varepsilon w_\varepsilon)=\max_{t \geq 0}\tilde{\Phi}_\varepsilon(t w_\varepsilon).
\end{equation}

By a standard argument, we know that there exist constants $T_1,T_2>0$ such that $T_1\leq t_\varepsilon\leq T_2$. So, we can assume that $t_\varepsilon\rightarrow t_0>0$ as $\varepsilon\rightarrow 0$. Note that, as $\varepsilon\rightarrow 0$
\begin{equation*}
\int_{\mathbb{R}^3}V(\varepsilon x)w_\varepsilon^2dx\rightarrow V_0\int_{\mathbb{R}^3}w^2dx,~\int_{\mathbb{R}^3}f(t_\varepsilon w_\varepsilon)wdx\rightarrow \int_{\mathbb{R}^3}f(t_0 w)wdx
\end{equation*}
and
\begin{equation*}
\int_{\mathbb{R}^3}
\big[\frac{1}{|x|}*K(\varepsilon x)w_\varepsilon^2\big]K(\varepsilon x)w_\varepsilon^2dx\rightarrow K^2(0)t_0^4\int_{\mathbb{R}^3}
\big[\frac{1}{|x|}*w^2\big]w^2dx=0
\end{equation*}
which implies that $t_0w\in \mathcal{N}^{V_0}$. Moreover, $w$ is also in $\mathcal{N}^{V_0}$. So, $t_0w=w$ and hence $t_0=1$ by the uniqueness. Therefore, it follows from Lemma \ref{Lem3.3} and \eqref{eq3.9} that
\begin{equation*}
\limsup\limits_{\varepsilon\rightarrow 0}\tilde{c}_\varepsilon\leq \limsup\limits_{\varepsilon\rightarrow 0}\max_{t\geq 0}\tilde{\Phi}_\varepsilon(t w_\varepsilon)=\limsup\limits_{\varepsilon\rightarrow 0}\tilde{\Phi}_\varepsilon(t_\varepsilon w_\varepsilon)=\mathcal{J}_{V_0}(w)=c_{V_0}.
\end{equation*}
This completes the proof.
\end{proof}

\begin{Lem}\label{Lem3.8}
For any $\varepsilon_n\rightarrow 0$, consider the sequence $\{y_{\varepsilon_n}\}\subset \mathbb{R}^3$ given in Lemma \ref{Lem3.4} and $v_n(x):=u_{\varepsilon_n}(x+y_{\varepsilon_n})$, with $u_{\varepsilon_n}$ obtained in Lemma \ref{Lem3.6}. Then, there is $u\in H^1(\mathbb{R}^3)\setminus \{0\}$ such that, up to a subsequence
\begin{equation*}
v_n\rightarrow u~\text{in}~H^1(\mathbb{R}^3).
\end{equation*}
Moreover, there is $x_0\in \Lambda$ such that
\begin{equation*}
\lim_{n\rightarrow \infty}\varepsilon_n y_{\varepsilon_n}=x_0\quad\text{and}\quad V(x_0)=V_0.
\end{equation*}
\end{Lem}
\begin{proof}
For simplicity of notations, we denote $u_{\varepsilon_n}$ and $y_{\varepsilon_n}$ by $u_n$ and $y_n$, respectively. We will divided the proof into the following three claims.\\
{\bf Claim 1:} $\{\varepsilon_n y_n\}$ is bounded.
\par
In fact, suppose to the contrary that, up to a subsequence, $|\varepsilon_n y_n|\rightarrow\infty$ as $n\rightarrow\infty$. We can assume that $V(\varepsilon_n x+\varepsilon_n y_n)\rightarrow V_\infty, K(\varepsilon_n x+\varepsilon y_n)\rightarrow K_\infty$ and $\chi(\varepsilon_n x+\varepsilon_n y_n)\rightarrow 0$ as $n\rightarrow \infty$ uniformly hold on bounded sets of $x\in \mathbb{R}^3$, where $V_\infty\geq V_0$ and $K_{\infty}\geq 0$. Similar to Lemma \ref{Lem3.2}, $\{u_n\}$ is bounded in $H^1(\mathbb{R}^3)$ and so $\{v_n\}$ is also bounded in $H^1(\mathbb{R}^3)$. Hence, there is nonnegative $u\in H^1(\mathbb{R}^3)\setminus \{0\}$ such that
\begin{equation*}
v_n\rightharpoonup u~~\text{in}~H^1(\mathbb{R}^3).
\end{equation*}
Obviously, $v_n$ satisfies
\begin{equation*}
-\Delta v_n+V(\varepsilon_n x+\varepsilon_n y_n)v_n+(\frac{1}{|x|}*K(\varepsilon_n x+\varepsilon_n y_n)v_n^2)K(\varepsilon_n x+\varepsilon_n y_n)v_n=g(\varepsilon_n x+\varepsilon_n y_n,v_n).
\end{equation*}
Thus, for any $\varphi\in C_0^\infty(\mathbb{R}^3)$, there holds
\begin{equation*}
\begin{split}
0=&\int_{\mathbb{R}^3}\big(\nabla v_n \nabla \varphi+V(\varepsilon_n x+\varepsilon_n y_n)v_n \varphi)dx+\int_{\mathbb{R}^3}
(\frac{1}{|x|}*K(\varepsilon_n x+\varepsilon_n y_n)v_n^2)K(\varepsilon_n x+\varepsilon_n y_n)v_n \varphi dx\\
&-\int_{\mathbb{R}^3}g(\varepsilon_n x+\varepsilon_n y_n,v_n)\varphi dx.
\end{split}
\end{equation*}
Since $V$ is continuous and bounded, one has
\begin{equation*}
\int_{\mathbb{R}^3}V(\varepsilon_n x+\varepsilon_n y_n)v_n \varphi dx\rightarrow V_\infty\int_{\mathbb{R}^3}u \varphi dx,~\int_{\mathbb{R}^3}g(\varepsilon_n x+\varepsilon_n y_n,v_n)\varphi dx\rightarrow \int_{\mathbb{R}^3}f_*(u)\varphi dx.
\end{equation*}
We claim that
\begin{equation*}
\int_{\mathbb{R}^3}
(\frac{1}{|x|}*K(\varepsilon_n x+\varepsilon_n y_n)v_n^2)K(\varepsilon_n x+\varepsilon_n y_n)v_n \varphi dx\rightarrow K_\infty^2\int_{\mathbb{R}^3}
\big[\frac{1}{|x|}*u^2\big]u \varphi dx.
\end{equation*}
Note that
\begin{equation*}
\begin{split}
&\int_{\mathbb{R}^3}
(\frac{1}{|x|}*K(\varepsilon_n x+\varepsilon_n y_n)v_n^2)K(\varepsilon_n x+\varepsilon_n y_n)v_n \varphi dx-K_\infty^2\int_{\mathbb{R}^3}
\big[\frac{1}{|x|}*u^2\big]u \varphi dx\\
&=\int_{\mathbb{R}^3}
(\frac{1}{|x|}*K(\varepsilon_n x+\varepsilon_n y_n)v_n^2)\big(K(\varepsilon_n x+\varepsilon_n y_n)v_n-K_\infty u\big)\varphi dx\\
&\quad+K_\infty\int_{\mathbb{R}^3}
\bigg[\frac{1}{|x|}*\big(K(\varepsilon_n x+\varepsilon_n y_n)v_n^2-K_\infty u^2\big)\bigg]u\varphi dx.
\end{split}
\end{equation*}
For the first term, by Lemma \ref{Lem2.2}, one has
\begin{equation*}
\begin{split}
&\bigg|\int_{\mathbb{R}^3}
(\frac{1}{|x|}*K(\varepsilon_n x+\varepsilon_n y_n)v_n^2)\big(K(\varepsilon_n x+\varepsilon_n y_n)v_n-K_\infty u\big)\varphi dx\bigg|\\
&\leq \|K\|_{{\infty}}\|v_n\|_{\frac{12}{5}}^2\|(K(\varepsilon_n x+\varepsilon_n y_n)v_n-K_\infty u)\varphi\|_{\frac{6}{5}}\\
&\leq C\bigg(\int_{\mathbb{R}^3}|K(\varepsilon_n x+\varepsilon_n y_n)v_n-K_\infty u|^{\frac{6}{5}}|\varphi|^{\frac{6}{5}}dx\bigg)^{\frac{6}{5}}.
\end{split}
\end{equation*}
It is easy to check that
\begin{equation*}
\bigg(\int_{\mathbb{R}^3}|K(\varepsilon_n x+\varepsilon_n y_n)v_n-K_\infty u|^{\frac{6}{5}}|\varphi|^{\frac{6}{5}}dx\bigg)^{\frac{6}{5}}\rightarrow 0.
\end{equation*}
and then
\begin{equation*}
\bigg|\int_{\mathbb{R}^3}
(\frac{1}{|x|}*K(\varepsilon_n x+\varepsilon_n y_n)v_n^2)\big(K(\varepsilon_n x+\varepsilon_n y_n)v_n-K_\infty u\big)\varphi dx\bigg|\rightarrow 0.
\end{equation*}
For the second term, since $K(\varepsilon_n x+\varepsilon_n y_n)v_n^2\rightarrow K_\infty u^2$ a.e. $x\in \mathbb{R}^3$, then
\begin{equation*}
K(\varepsilon_n x+\varepsilon_n y_n)v_n^2\rightharpoonup K_{\infty}u^2~\text{in}~L^{\frac{6}{5}}(\mathbb{R}^3).
\end{equation*}
Recall that the convolution operator
\begin{equation*}
\frac{1}{|x|}*w(x) \in L^6(\mathbb{R}^3)
\end{equation*}
for all $w\in L^{\frac{6}{5}}(\mathbb{R}^3)$ and it is a linear bounded operator from $L^{\frac{6}{5}}(\mathbb{R}^3)$ to $L^6(\mathbb{R}^3)$. Consequently
\begin{equation*}
\frac{1}{|x|}*\big(K(\varepsilon_n x+\varepsilon_n y_n)v_n^2\big)\rightharpoonup K_{\infty}\big(\frac{1}{|x|}*u^2\big)~\text{in}~L^6(\mathbb{R}^3).
\end{equation*}
So,
\begin{equation*}
\int_{\mathbb{R}^3}
\bigg[\frac{1}{|x|}*\big(K(\varepsilon_n x+\varepsilon_n y_n)v_n^2-K_\infty u^2\big)\bigg]u\varphi dx\rightarrow 0.
\end{equation*}
Hence $u$ satisfies
\begin{equation*}
-\Delta u+V_\infty u+K_{\infty}^2\big(\frac{1}{|x|}*u^2\big)u=f_*(|u|)u,~x\in \mathbb{R}^3.
\end{equation*}
Taking the scalar product of this equation with $u$, one has
\begin{equation*}
\begin{split}
0&=\int_{\mathbb{R}^3}|\nabla u|^2dx+V_\infty\int_{\mathbb{R}^3}u^2dx+K_{\infty}^2\int_{\mathbb{R}^3}\big(\frac{1}{|x|}*u^2\big)u^2dx-\int_{\mathbb{R}^3}f_*(u)udx\\
&\geq\int_{\mathbb{R}^3}|\nabla u|^2dx+V_0\int_{\mathbb{R}^3}u^2dx-\frac{V_0}{k}\int_{\mathbb{R}^3}u^2dx\\
&\geq \int_{\mathbb{R}^3}|\nabla u|^2dx+\frac{\kappa-1}{\kappa}V_0\int_{\mathbb{R}^3}u^2dx,
\end{split}
\end{equation*}
which is a contradiction with $u\neq 0$.
\par
After extracting a subsequence, we may assume $\varepsilon_n y_n\rightarrow x_0$ as $n\rightarrow \infty$. If $x_0\notin \bar{\Lambda}$, then there exists $\delta_0>0$ such that $\{\varepsilon_n y_n\}\subset \mathbb{R}^3\setminus \Lambda^{\delta_0}$ for $n$ large enough. Repeating the argument in the proof of Claim 1, we can get a same contradiction. Thus, we have that $x_0\in \bar{\Lambda}$.\\
{\bf Claim 2:}
$x_0\in \Lambda$.
\par
A similar argument as discussed in Lemma \ref{Lem3.6}, we know that $u$ is a solution of the following equation
\begin{equation*}
-\Delta u+V(x_0)u+K^2(x_0)\big(\frac{1}{|x|}*u^2\big)u=\bar{g}(x,u),~x\in \mathbb{R}^3
\end{equation*}
where
\begin{equation*}
\bar{g}(x,u):=\xi(x)f(u)+(1-\xi(x))f_*(u)
\end{equation*}
and
\begin{equation*}
\chi(\varepsilon_n x+\varepsilon_ny_n)\rightarrow \xi(x)~a.e.~\text{in}~\mathbb{R}^3,
\end{equation*}
with energy
\begin{equation*}
\check{\Phi}_{x_0}(u)=\frac{1}{2}\int_{\mathbb{R}^3}|\nabla u|^2dx+\frac{V(x_0)}{2}\int_{\mathbb{R}^3}u^2dx+\frac{K^2(x_0)}{4}\int_{\mathbb{R}^3}\big(\frac{1}{|x|}*u^2\big)u^2dx
-\int_{\mathbb{R}^3}G(x,u)dx,
\end{equation*}
and the corresponding mountain pass level by $\check{c}_{x_0}$. Note that, $G(x, t)\leq F(t)$ for all $x\in \mathbb{R}^3, t\geq 0$, which gives then $c_{x_0}\leq \check{c}_{x_0}$, where $c_{x_0}$ denotes the mountain pass level associated with $\check{J}_{x_0}$, where $\check{J}_{x_0}:E\rightarrow \mathbb{R}$ is given by
\begin{equation*}
\check{J}_{x_0}(u)=\frac{1}{2}\int_{\mathbb{R}^3}|\nabla u|^2dx+\frac{V(x_0)}{2}\int_{\mathbb{R}^3}u^2dx+\frac{K^2(x_0)}{4}\int_{\mathbb{R}^3}\big(\frac{1}{|x|}*u^2\big)u^2dx
-\int_{\mathbb{R}^3}F(u)dx.
\end{equation*}
Then,
\begin{equation*}
\begin{split}
c_{x_0}&\leq \check{c}_{x_0}\leq \check{\Phi}_{x_0}(u)= \check{\Phi}_{x_0}(u)-\frac{1}{4}\check{\Phi}_{x_0}^\prime(u)u\\
&=\frac{1}{4}\int_{\mathbb{R}^3}| \nabla u|^2dx+\frac{1}{4}(1-\frac{1}{\kappa})\int_{\mathbb{R}^3} V(x_0)u^2dx+\int_{\mathbb{R}^3}\bigg(\frac{1}{4\kappa}V(x_0)u^2+\frac{1}{4}\check{g}(x,u)u-\check{G}(x,u)\bigg)dx\\
&\leq\liminf_{n\rightarrow\infty}\bigg[\frac{1}{4}\int_{\mathbb{R}^3}| \nabla v_n|^2dx +\frac{1}{4}(1-\frac{1}{\kappa})\int_{\mathbb{R}^3}V(\varepsilon_n x+\varepsilon_ny_n)v_n^2dx\\
&\quad+\int_{\mathbb{R}^3}\bigg(\frac{1}{4\kappa}V(\varepsilon_n x+\varepsilon_ny_n)v_n^2+\frac{1}{4}g(\varepsilon_n x+\varepsilon_ny_n,v_n)v_n-G(\varepsilon_n x+\varepsilon_ny_n,v_n)\bigg)dx\bigg]\\
&=\liminf_{n\rightarrow\infty} c_{\varepsilon_n}\leq c_{V_0}.
\end{split}
\end{equation*}
Therefore, $c_{x_0}\leq c_{V_0}$, which implies that $V(x_0)\leq V(0)=V_0$. Then, by it follows that $x_0\notin \partial \Lambda$ and so $x_0\in \Lambda$, proving the claim.\\
{\bf Claim 3:}
$v_n\rightarrow u$ in $H^1(\mathbb{R}^3)$ as $n\rightarrow \infty$.
\par
From Lemma \ref{Lem3.5}, it is easy to see that
\begin{equation*}
\lim_{n\rightarrow\infty}\int_{\mathbb{R}^3}| \nabla u_n|^2dx=\int_{\mathbb{R}^3}| \nabla u|^2dx
\end{equation*}
and
\begin{equation*}
\lim_{n\rightarrow\infty}\int_{\mathbb{R}^3}V(\varepsilon_n x+\varepsilon_ny_n)v_n^2dx=V(x_0)\int_{\mathbb{R}^3}u^2dx.
\end{equation*}
The last limit gives
\begin{equation*}
\lim_{n\rightarrow\infty}\int_{\mathbb{R}^3}v_n^2dx=\int_{\mathbb{R}^3}u^2dx.
\end{equation*}
Thus,
\begin{equation*}
\lim_{n\rightarrow\infty}\big(\int_{\mathbb{R}^3}|\nabla v_n|^2dx+\int_{\mathbb{R}^3}v_n^2dx\big)=\int_{\mathbb{R}^3}| \nabla u|^2dx+\int_{\mathbb{R}^3}u^2dx.
\end{equation*}
Together with $v_n\rightharpoonup u$ in $H^1(\mathbb{R}^3)$, we have $v_n\rightarrow u$ in $H^1(\mathbb{R}^3)$.
\end{proof}
\par
The following lemma plays a fundamental role in the study of behavior of the maximum points of the
solutions. We omit its proof, which are related to the Moser iterative method \cite{He2011zamp,YangNA20}.
\begin{Lem}\label{Lem3.8}
There exists $C>0$ independent of $n$ such that $\|v_n\|_\infty\leq C$ and
\begin{equation*}
\lim_{|x|\rightarrow\infty}v_n(x)=0~\text{uniformly~in}~n \in \mathbb{N}.
\end{equation*}
Furthermore, there exist $c_1,c_2>0$ such that
\begin{equation*}
v_n(x)\leq c_1e^{-c_2|x|},~\forall x\in \mathbb{R}^3.
\end{equation*}
\end{Lem}

\section{Proof of main results}
In this section we will prove our main result. The idea is to show that the solutions obtained in Lemma \ref{Lem3.6} verify the following estimate $u_\varepsilon(x)\leq a $, $\forall x\in\Lambda^c_\varepsilon$ for $\varepsilon$ small enough. This fact implies that these solutions are in fact solutions of the original problem \eqref{eq2.1}.
\begin{Lem}\label{Lem4.1}
There exists $n_0\in \mathbb{N}$ such that
\begin{equation*}
u_n\leq a,~\forall n\geq n_0~\text{and}~x\in \Lambda_{\varepsilon_n}^c.
\end{equation*}
Hence, $u_n$ is a solution of \eqref{eq2.1} with $\varepsilon=\varepsilon_n$ for $n\geq n_0$.
\end{Lem}
\begin{proof}
By Lemma \ref{Lem3.8}, we obtain that $\varepsilon_n y_n\rightarrow x_0\in \Lambda$. Thus, up to a subsequence, there exists $r>0$ such that
\begin{equation*}
B(\varepsilon_n y_n,r)\subset \Lambda,~\forall n\in \mathbb{N}.
\end{equation*}
Hence,
\begin{equation*}
B(y_n,\frac{r}{\varepsilon_n })\subset \Lambda_{\varepsilon_n},~\forall n\in \mathbb{N},
\end{equation*}
which implies that
\begin{equation*}
\Lambda_{\varepsilon_n}^c\subset \mathbb{R}^3\setminus B(y_n,\frac{r}{\varepsilon_n }),~\forall n\in \mathbb{N}.
\end{equation*}
Next, by Lemma \ref{Lem3.8}, there exists $R>0$ such that
\begin{equation*}
v_n(x)\leq a,~\text{for}~|x|\geq R~\text{and}~\forall n\geq n_0,
\end{equation*}
which means that
\begin{equation*}
u_n(x)=v_n(x-y_n)\leq a,~x\in \mathbb{R}^3\setminus B(y_n,R)~\text{and}~\forall n\geq n_0.
\end{equation*}
On the other hand, there is some $n_0\in \mathbb{N}$ such that
\begin{equation*}
\Lambda_{\varepsilon_n}^c\subset \mathbb{R}^3\setminus B(y_n,\frac{r}{\varepsilon_n })\subset \mathbb{R}^3\setminus B(y_{\varepsilon_n},R),~\forall n\geq n_0.
\end{equation*}
Therefore,
\begin{equation*}
u_n\leq a,~\forall n\geq n_0~\text{and}~x\in \Lambda_{\varepsilon_n}^c.
\end{equation*}
This completes the proof.
\end{proof}

\textbf{Proof of Theorem \ref{Thm1.1}.} By Lemma \ref{Lem3.6}, the problem \eqref{eq3.3} has a ground state solution $u_\varepsilon$ for all $\varepsilon>0$. From Lemma \ref{Lem4.1}, there exists $\varepsilon_0>0$ such that
\begin{equation*}
u_\varepsilon\leq a,~\forall x\in \mathbb{R}^3\setminus\Lambda_\varepsilon~\text{and}~\forall \varepsilon\in (0,\varepsilon_0).
\end{equation*}
that is, $u_\varepsilon$ is a solution of \eqref{eq2.1} for $\varepsilon\in (0,\varepsilon_0)$. Considering
\begin{equation*}
\omega_\varepsilon(x)=u_\varepsilon(\frac{x}{\varepsilon}),~\forall \varepsilon\in (0,\varepsilon_0).
\end{equation*}
Then $\omega_\varepsilon$ is a solution of the original system \eqref{eq1.1}.
\par
Now, we claim that there exists a $\rho_0 > 0$ such that $\|v_n\|_\infty\geq \rho_0,\forall n\in \mathbb{N}$. In fact, suppose that $\|v_n\|\rightarrow 0$ as $n \rightarrow\infty $. Then there exists $n_0\in \mathbb{N}$ such that
\begin{equation*}
g(\varepsilon_n x+\varepsilon_n y_n,v_n)\leq \frac{V_0}{2}v_n~\text{for}~n\geq n_0.
\end{equation*}
Therefore, we have
\begin{equation*}
\begin{split}
\int_{\mathbb{R}^3}|\nabla v_n|^2dx+V_0\int_{\mathbb{R}^3}v_n^2dx&\leq\int_{\mathbb{R}^3}|\nabla v_n|^2dx+\int_{\mathbb{R}^3}V(\varepsilon_n x+\varepsilon_ny_n)v_n^2dx\\
&\quad+\int_{\mathbb{R}^3}
(\frac{1}{|x|}*K(\varepsilon_n x+\varepsilon_n y_n)v_n^2)K(\varepsilon_n x+\varepsilon_n y_n)v_n^2dx\\
&=\int_{\mathbb{R}^3}g(\varepsilon_n x+\varepsilon_n y_n)v_ndx\\
&\leq \frac{V_0}{2}\int_{\mathbb{R}^3}v_n^2dx.
\end{split}
\end{equation*}
This implies that $\|v_n\|_{V_0}=0$ for $n \geq n_0$, which is impossible because $v_n\rightarrow u$ in $H^1(\mathbb{R}^3)$ and $u\neq 0$. Then, the claim is true.
\par
From the above claim, we see that $v_n$ has a global maximum points $p_n\in B_{R_0}(0)$ for some $R_0>0$. Hence, the global maximum points of $\omega_n$ given by $x_n:=\varepsilon_n(p_n+ y_n)$. Since $p_n\in B_{R_0}(0)$ is bounded, then we know that $\{x_n\}$ is bounded and $x_n\rightarrow x_0 \in \Lambda$, that is
\begin{equation*}
\lim_{n\rightarrow\infty}V(x_n)=V(x_0).
\end{equation*}
Furthermore, similar to the argument of Lemma \ref{Lem3.8}, $\tilde{v}_n(x):=\omega_n(\varepsilon_n x+\varepsilon_n x_n)$ converges to a positive ground state solution of
\begin{equation*}
-\Delta u+V_0u=f(u).
\end{equation*}
At last, by Lemma \ref{Lem3.8}, one has
\begin{equation*}
\begin{split}
\omega_n(x)=v_n(\frac{x}{\varepsilon_n}-y_n)
\leq Ce^{-c|\frac{x-\varepsilon_n y_n}{\varepsilon_n}|}\leq Ce^{-c|\frac{x-\varepsilon_n y_n-\varepsilon_n p_n}{\varepsilon_n}|}=Ce^{-\frac{c}{\varepsilon_n}|x-x_n|}.
\end{split}
\end{equation*}
Thus, the proof of Theorem \ref{Thm1.1} is completed.

\section{Some open questions}
We close this paper by listing several problems that are left open in this direction.\\

{\bf Problem 1: Uniqueness and Non-degeneracy}
\par
If we set $f(u)=|u|^{p-2}u$ in our problem \eqref{eq1.1} with $4<p<6$, we have
\begin{equation*}
\begin{cases}
-\varepsilon^2\Delta u+V(x)u+K(x)\phi u=|u|^{p-2}u&\text{in}\ \mathbb{R}^3,\\
-\varepsilon^2\Delta \phi=K(x)u^2&\text{in}\ \mathbb{R}^3,
\end{cases}
\end{equation*}
which has a ground state solution converging to the uniqueness ground state solution of
\begin{equation}\label{eq5.1}
-\Delta u+V_0u=|u|^{p-2}u\quad \text{in}~\mathbb{R}^3.
\end{equation}
For \eqref{eq5.1}, uniqueness and non-degeneracy has a long history and has been addressed by many authors, see \cite{Lin91CPDE,Coffman72ARMA,Kwong89ARMA}. Then we want to know whether the Schr\"{o}dinger-Poisson system also has similar results.\\

{\bf Problem 2: Concentration phenomenon for local competing potential functions}
\par
In this problem we want to consider the following problem:
\begin{equation*}
\begin{cases}
-\varepsilon^2\Delta u+V(x)u+\phi u=Q(x)f(u)&\text{in}\ \mathbb{R}^3,\\
-\varepsilon^2\Delta \phi=u^2&\text{in}\ \mathbb{R}^3.
\end{cases}
\end{equation*}
The question is about the concentration phenomenon when $V(x)$ has local minimum and $Q(x)$ has a local maximum:
\begin{itemize}
\item[$(a)$] $V,Q\in C(\R^3,\R)$ and $\inf\limits_{x\in\R^3}V(x)=V_1>0$, $\inf\limits_{x\in\R^3}Q(x)=Q_1>0$.
\item[$(b)$] There is a open and bounded domain $\Lambda$ such that
\begin{equation*}
V(x_{min}):=\inf_{\Lambda}V(x)<\min_{\partial\Lambda}V(x),\quad Q(x_{max})=\sup_{\Lambda}Q(x)>\max_{\partial\Lambda}Q(x).
\end{equation*}
\end{itemize}
In this case, $V$ and $Q$ all want to attract solutions to their local extreme point, the we call it local competing potential functions. In fact, if the condition is global, there have been some results for Schr\"{o}dinger-Poisson system \cite{Wang13CV,Wang15ZAMP,YangNA20} and fractional Schr\"{o}dinger-Poisson system \cite{Yang19MN,Yu17CV}.\\

{\bf Problem 3: Double-critical case}
\par
For the double critical Schr\"{o}dinger-Poisson system
\begin{equation*}
\begin{cases}
-\Delta u+V(x)u+K(x)\phi|u|^3u=|u|^{4}u&\text{in}\ \mathbb{R}^3,\\
-\Delta \phi=K(x)|u|^5&\text{in}\ \mathbb{R}^3,
\end{cases}
\end{equation*}
or more general
\begin{equation}\label{eq5.2}
-\Delta u+V(x)u+(I_\alpha*K(x)|u|^{p})K(x)\phi|u|^{p-2}u=|u|^{q-2}u\quad\text{in}\ \mathbb{R}^N,
\end{equation}
where $1<p\leq\frac{N+\alpha}{N-2}, 2<q\leq2^*=\frac{2N}{N-2}$ and $I_{\alpha}: \mathbb{R}^{N} \rightarrow \mathbb{R}$ is the Riesz potential of order $\alpha \in(0, N)$ defined for $x \in \mathbb{R}^{N} \backslash\{0\}$ as
\[
I_{\alpha}(x)=\frac{A_{\alpha}}{|x|^{N-\alpha}}, \quad A_{\alpha}=\frac{\Gamma\left(\frac{N-\alpha}{2}\right)}{\Gamma\left(\frac{\alpha}{2}\right) \pi^{N / 2} 2^{\alpha}}.
\]
Because for $\alpha \in(0, N)$ the Riesz potential $I_{\alpha}$ is the Green function of the fractional Laplacian $(-\Delta)^{\alpha / 2}$, then the system
\[
\left\{\begin{array}{l}
-\Delta u+\phi|u|^{p-2} u=|u|^{q-2} u, \\
(-\Delta)^{\alpha / 2} \phi=u^{p},
\end{array}\right.
\]
is formally equivalent to equation \eqref{eq5.2}. Therefore, we want to know the existence, multiplicity and different asymptotic behavior depending on $\alpha\rightarrow0$.

\section*{Acknowledgments}
We would like to thank the anonymous referee for his/her careful readings of our manuscript and the useful comments.\\
{\bf Data availability:}
Data sharing not applicable to this article as no datasets were generated or analysed during the current study.

\bibliographystyle{plain}
\bibliography{Yang-Yu}

\end{document}